\documentclass[11pt]{article}

\usepackage[margin=1.1in]{geometry}
\usepackage{amsmath,amssymb,amsthm,mathtools}
\usepackage{enumitem}
\usepackage{hyperref}
\usepackage{xcolor}

\hypersetup{
  colorlinks=true,
  linkcolor=blue!60!black,
  citecolor=blue!60!black,
  urlcolor=blue!60!black
}

\newtheorem{theorem}{Theorem}[section]
\newtheorem{lemma}[theorem]{Lemma}
\newtheorem{proposition}[theorem]{Proposition}
\newtheorem{corollary}[theorem]{Corollary}
\theoremstyle{definition}
\newtheorem{definition}[theorem]{Definition}
\theoremstyle{remark}

\newcommand{\rlt}{\rho_{<}}
\newcommand{\exlt}{\operatorname{ex}_{<}}
\newcommand{\chilt}{\chi_{<}}
\newcommand{\pilt}{\vec\pi}
\newcommand{\Pvec}{\vec P}
\newcommand{\eps}{\varepsilon}

\title{Binary smoothing and relative Tur\'an densities of ordered triangle-tails}
\author{Shuyan Chen\\
Department of Mathematics, The University of Manchester\\
\texttt{shuyan.chen-2@student.manchester.ac.uk}}
\date{}

\begin{document}
\maketitle

\begin{abstract}
For every $b\ge1$, let $Q_{2,b}$ be the ordered graph obtained from a transitive ordered triangle by attaching a monotone tail of length $b$ at its rightmost vertex.  We prove
\[
\rho_{<}(Q_{2,b})=\frac12\qquad (b\ge1).
\]
Thus the previously isolated case $Q_{2,2}$ is one member of an exact infinite triangle-tail family.  The lower bound is the sharp forward-template identity $\lambda(Q_{2,b})=1/2$, and it is realized already inside binary-level hosts: a parity-cut construction gives $Q_{2,b}$-free binary-level graphs with density exactly $1/2$ on every level.  The upper bound uses the binary rich-level reduction.  Its key input is the intrinsic decomposition of a $Q_{2,b}$-free graph into tail-starting vertices $T_b$ and the complement $R$: there are no forward edges from $R$ to $T_b$, $G[T_b]$ is ordered-triangle-free, and $G[R]$ is monotone-$\vec P_{b+1}$-free.  We control these pieces by weighted binary $\vec P_{b+1}$ smoothing and weighted binary Mantel smoothing, the latter following from a binary ultrametric cut-domination theorem.  We also record exact path-blow-up values, giving a reusable template/rich-host calculus for ordered relative densities.
\end{abstract}

\section{Introduction}

For an ordered graph $F$ and an ordered host graph $G$, let
\[
\exlt(G,F)=\max\{e(H):H\subseteq G,\ H\text{ contains no ordered copy of }F\}.
\]
The relative Tur\'an density of $F$ is
\[
\rlt(F)=\inf_G \frac{\exlt(G,F)}{e(G)},
\]
where the infimum is over all non-empty ordered host graphs.  Equivalently, $\rlt(F)$ is the largest edge proportion that can always be retained while avoiding $F$.

Classical Tur\'an theory begins with Mantel's triangle theorem and Tur\'an's clique theorem, while the Erd\H{o}s--Stone--Simonovits theorem determines the asymptotic density for every non-bipartite forbidden graph~\cite{Mantel07,Turan41,ES46,ES66}.  For ordered graphs, Pach and Tardos proved the corresponding complete-host formula in terms of interval chromatic number~\cite{PT06}.  A substantial parallel literature studies ordered pattern avoidance, including the matrix viewpoint of Marcus and Tardos~\cite{MT04}, ordered forests~\cite{KTTW19}, and bipartite ordered graphs~\cite{MT22}; see also the surveys of Tardos~\cite{TardosICM,Tardos19}.

The relative ordered Tur\'an density was introduced by Reiher, R\"odl, Sales and Schacht, who established its basic theory, determined monotone paths and cliques, and proved blow-up invariance~\cite{RRSS25}.  King, Lidick\'y, Ouyang, Pfender, Wang and Xiang produced the first intermediate-density examples and introduced the chorded paths $Q_{a,b}$ studied here~\cite{KLOPWX25}.  Illingworth, Ranganathan, Versteegen and Williams subsequently proved that binary-level hosts are universal and characterized the zero-density case~\cite{IRVW25}.  The sparse-host constructions are tied to shift graphs, a classical source of almost bipartite graphs with large chromatic number~\cite{EHS82}; the relevant independent-set estimates were developed by Arman, R\"odl and Sales~\cite{ARS22}.

For $b\ge1$, let $Q_{2,b}$ be the ordered graph on
\[
1<2<\cdots<b+3
\]
with edge set
\[
12,\quad 13,\quad 23,\quad 34,\quad45,\quad\ldots,\quad(b+2)(b+3).
\]
Thus $Q_{2,b}$ is a transitive ordered triangle whose rightmost vertex starts a monotone tail of length $b$.  Since every consecutive pair is an edge, its interval chromatic number is $b+3$, and therefore
\[
\vec\pi(Q_{2,b})=1-\frac1{b+2}
\]
by Pach--Tardos.  King et al.~proved general bounds for the family $Q_{a,b}$ and posed the problem of determining its relative densities exactly~\cite{KLOPWX25}.  The relative density is much smaller.

\begin{theorem}[Main theorem]\label{thm:main}
For every $b\ge1$,
\[
\rlt(Q_{2,b})=\frac12.
\]
\end{theorem}

For $b=2$, this closes the previously known gap from~\cite{KLOPWX25}
\[
\frac12\le \rlt(Q_{2,2})\le \frac23<\vec\pi(Q_{2,2})=\frac34.
\]
The theorem does not merely settle the first non-trivial chorded path.  It identifies an infinite ordered triangle-tail family whose relative density collapses to the bipartite barrier $1/2$, independent of the tail length.

The constant $1/2$ is sharp in the same binary rich-level language used for the upper bound.  For every depth $d\ge2$, a parity cut of the binary interval $B_d$, with half of the deepest sibling edges retained, is bipartite and hence $Q_{2,b}$-free, while every binary level has density exactly $1/2$.  Therefore no forcing statement of the Illingworth--Ranganathan--Versteegen--Williams type can have threshold below $1/2$.

The proof has two independent sides.  The lower bound is a forward-template obstruction: we define a parameter $\lambda(F)$, prove $\rlt(F)\ge\lambda(F)$ for every ordered graph $F$, and show
\[
\lambda(Q_{2,b})=\frac12\qquad(b\ge1).
\]
This gives the right lower bound for the whole family and explains why the answer is the bipartite obstruction, not the complete-host density.

The upper bound uses the binary rich-level reduction.  For every fixed $b\ge1$ and every $\eps>0$, we prove that there is $C=C(b,\eps)$ such that every binary-level graph with at least $C$ levels of density at least $1/2+\eps$ contains $Q_{2,b}$.  The structural split is forced by the graph itself.  Let $T_b(G)$ be the set of vertices that start a monotone tail of length $b$, and let $R=V(G)\setminus T_b(G)$.  In a $Q_{2,b}$-free graph there are no forward edges from $R$ to $T_b(G)$, the graph induced by $T_b(G)$ is ordered-triangle-free, and the graph induced by $R$ is monotone-$\Pvec_{b+1}$-free.  The $R$-part is controlled by a $b$-rank smoothing lemma; the $T$-part is controlled by a weighted Mantel theorem on binary levels.

The main technical input for the Mantel step is a binary-ultrametric weighted analogue of the classical Mantel--Tur\'an extremal theorem~\cite{Mantel07,Turan41}.  It is proved by a mass-sensitive Zykov-type cherry induction, in the spirit of classical symmetrization~\cite{Zykov49}.

\begin{theorem}[Binary ultrametric cut domination, informal]
Let $\mathcal U$ be a finite antichain of binary-tree nodes with positive masses, and weight each pair $XY$ by a non-negative function of the binary least common ancestor of $X$ and $Y$, multiplied by the product of the two masses.  Then every triangle-free graph on $\mathcal U$ has weight at most the maximum cut weight of this weighted complete graph.
\end{theorem}

This theorem implies weighted binary Mantel smoothing: an ordered-triangle-free graph on a subset of a binary interval cannot beat the level-by-level half-baseline on more than boundedly many levels.  Together with weighted binary $\Pvec_{b+1}$ smoothing and the intrinsic role split, this proves Theorem~\ref{thm:main}.

We also prove an auxiliary exact family.  If $F$ is sandwiched between a monotone path $\Pvec_k$ and an ordered blow-up of $\Pvec_k$, then
\[
\rlt(F)=\lambda(F)=\frac{k-2}{2(k-1)}.
\]
In particular, the hinged ordered matching family considered below has relative density $1/4$.

\subsection*{Organization}

Section~\ref{sec:prelim} recalls ordered relative densities and the binary rich-level reduction.  Section~\ref{sec:templates} proves the template lower bound and the path-blow-up exact family.  Section~\ref{sec:q2b} proves $\lambda(Q_{2,b})=1/2$.  Section~\ref{sec:roles} develops the intrinsic role decomposition for $Q_{2,b}$-free graphs.  Section~\ref{sec:smoothing} proves the finite-template and $b$-rank smoothing lemmas.  Section~\ref{sec:cutdom} proves binary ultrametric cut domination and weighted binary Mantel smoothing.  Section~\ref{sec:q2b-proof} combines these ingredients to prove Theorem~\ref{thm:main}.  Section~\ref{sec:sharpness} gives the binary-level sharpness construction.  Section~\ref{sec:local} records a local obstruction diagnostic for the first non-trivial member $Q_{2,2}$.

\section{Preliminaries}\label{sec:prelim}

An ordered graph is a graph with a fixed linear order on its vertices.  An ordered copy is an injective graph embedding preserving the vertex order.  We write $\Pvec_k$ for the monotone path on $k$ ordered vertices.

\begin{lemma}[Monotonicity]\label{lem:monotonicity}
If $F\subseteq_{<}H$, then $\rlt(F)\le \rlt(H)$.
\end{lemma}

\begin{proof}
For every host $G$, every $F$-free subgraph is $H$-free, since any copy of $H$ contains a copy of $F$.  Hence $\exlt(G,F)\le \exlt(G,H)$ for every $G$, and the claim follows by taking infima.
\end{proof}

\begin{theorem}[Pach--Tardos~\cite{PT06}]\label{thm:pach-tardos}
If $\chilt(F)$ is the interval chromatic number of $F$, then
\[
\pilt(F)=1-\frac1{\chilt(F)-1}.
\]
\end{theorem}

\begin{theorem}[Path densities and blow-up invariance~\cite{RRSS25}]\label{thm:path-known}
For every $k\ge2$,
\[
\rlt(\Pvec_k)=\frac{k-2}{2(k-1)}.
\]
Moreover, relative density is invariant under ordered blow-ups.
\end{theorem}

\begin{definition}[Binary intervals and levels]\label{def:binary-levels}
For $d\ge1$, the depth-$d$ binary interval is the ordered set
\[
B_d=\{0,1\}^d
\]
with lexicographic order.  If $x<y$ are two leaves, their binary level is
\[
\delta(x,y)=\min\{i:x_i\ne y_i\}\in[d].
\]
Equivalently, $x$ and $y$ lie in the two children of a common depth-$(\ell-1)$ fundamental interval, where $\ell=\delta(x,y)$.

A binary-level graph is an arbitrary graph on the ordered leaf set $B_d$.  Let
\[
E_\ell(B_d)=\{xy:x<y,\ \delta(x,y)=\ell\},
\qquad
\tau_{\ell,d}=|E_\ell(B_d)|.
\]
Since there are $2^{\ell-1}$ depth-$(\ell-1)$ fundamental intervals and each contributes a complete bipartite pair of parts of size $2^{d-\ell}$,
\[
\tau_{\ell,d}=2^{\ell-1}2^{2(d-\ell)}=2^{2d-\ell-1}.
\]
For a graph $G\subseteq K_{B_d}$, write
\[
e_\ell(G)=|E(G)\cap E_\ell(B_d)|,
\qquad
 d_\ell(G)=\frac{e_\ell(G)}{\tau_{\ell,d}}.
\]
A level $\ell$ is $\alpha$-rich if $d_\ell(G)\ge\alpha$, and $G$ is $(\alpha,C)$-rich if at least $C$ of its levels are $\alpha$-rich.
\end{definition}

\begin{theorem}[Binary rich-level reduction~\cite{IRVW25}]\label{thm:rich-level}
For an ordered graph $F$ and $\alpha\in[0,1]$, the inequality $\rlt(F)\le\alpha$ holds if and only if for every $\eps>0$ there exists $C=C(F,\eps)$ such that every binary-level graph that is $(\alpha+\eps,C)$-rich contains an ordered copy of $F$.
\end{theorem}

\begin{theorem}[King--Lidick\'y--Ouyang--Pfender--Wang--Xiang~\cite{KLOPWX25}]\label{thm:klopwx}
For the chorded paths $Q_{a,b}$ with $1\le b\le a$,
\[
\rlt(Q_{a,b})\le \frac{a}{a+1}.
\]
In particular, $\rlt(Q_{2,2})\le2/3$.
\end{theorem}

\section{Template lower bounds and path blow-ups}\label{sec:templates}

\begin{definition}[Forward-edge template]
A weighted forward-edge template is a triple $\mathcal T=(Q,\mu,A)$, where $Q$ is a finite label set, $\mu$ is a probability measure on $Q$, and $A\subseteq Q\times Q$ is the set of allowed forward edge label-pairs.  The template supports an ordered graph $F$ if there is a map $\phi:V(F)\to Q$ such that $(\phi(i),\phi(j))\in A$ for every ordered edge $i<j$ of $F$.  If no such map exists, the template is $F$-avoiding.  Define
\[
w(\mathcal T)=\sum_{(q,q')\in A}\mu(q)\mu(q'),
\qquad
\lambda(F)=\sup\{w(\mathcal T):\mathcal T\text{ is }F\text{-avoiding}\}.
\]
\end{definition}

\begin{theorem}[Template lower bound]\label{thm:template-lower}
For every ordered graph $F$,
\[
\rlt(F)\ge \lambda(F).
\]
\end{theorem}

\begin{proof}
Fix an ordered host $G$ and an $F$-avoiding template $(Q,\mu,A)$.  Label each vertex of $G$ independently according to $\mu$, and keep an ordered edge $uv$, $u<v$, exactly when the label-pair lies in $A$.  The expected retained edge proportion is $w(\mathcal T)$.  If the retained graph contained an ordered copy of $F$, the labels on that copy would support $F$, contradiction.  Thus some $F$-free subgraph of $G$ keeps at least $w(\mathcal T)e(G)$ edges.  Taking the infimum over $G$ and the supremum over $\mathcal T$ gives the claim.
\end{proof}

\begin{corollary}\label{cor:nonbip-lower}
If the underlying graph of $F$ is non-bipartite, then $\lambda(F)\ge1/2$ and hence $\rlt(F)\ge1/2$.
\end{corollary}

\begin{proof}
Use the two-label bipartite template with allowed pairs $(0,1)$ and $(1,0)$ and uniform measure.  It has weight $1/2$ and avoids every non-bipartite graph.
\end{proof}

\begin{corollary}[Path template]\label{cor:path-template}
For every $k\ge2$,
\[
\lambda(\Pvec_k)\ge \frac{k-2}{2(k-1)}.
\]
\end{corollary}

\begin{proof}
Take labels $[k-1]$ with the uniform measure and allow $(i,j)$ exactly when $i<j$.  The weight is $\binom{k-1}{2}/(k-1)^2=(k-2)/(2(k-1))$, and a supported monotone path on $k$ vertices would require $k$ strictly increasing labels in $[k-1]$.
\end{proof}

\begin{definition}[Ordered blow-up of a monotone path]
The ordered blow-up $\Pvec_k(m_1,\ldots,m_k)$ has consecutive independent blocks $V_1<\cdots<V_k$, $|V_i|=m_i$, and all edges between $V_i$ and $V_{i+1}$ for $1\le i<k$.
\end{definition}

\begin{theorem}[Path-blow-up sandwich]\label{thm:path-sandwich}
If
\[
\Pvec_k\subseteq_{<} F\subseteq_{<}\Pvec_k(m_1,\ldots,m_k),
\]
then
\[
\rlt(F)=\lambda(F)=\frac{k-2}{2(k-1)}.
\]
\end{theorem}

\begin{proof}
By monotonicity and Theorem~\ref{thm:path-known},
\[
\rlt(\Pvec_k)\le \rlt(F)\le \rlt(\Pvec_k(m_1,\ldots,m_k))=\rlt(\Pvec_k)=\frac{k-2}{2(k-1)}.
\]
The same path template from Corollary~\ref{cor:path-template} avoids $F$, because $\Pvec_k\subseteq_{<}F$.  Thus $\lambda(F)$ reaches the same value, and Theorem~\ref{thm:template-lower} gives equality.
\end{proof}

\begin{definition}[Hinged ordered matchings]
For $r,s\ge0$, define $H_{r,s}$ with vertex order
\[
a_0<a_1<\cdots<a_r<b_0<b_1<\cdots<b_{r+s}<c_0<c_1<\cdots<c_s
\]
and edge set
\[
\{a_0b_0,b_0c_0\}\cup\{a_i b_i:1\le i\le r\}\cup\{b_{r+j}c_j:1\le j\le s\}.
\]
\end{definition}

\begin{corollary}\label{cor:hinged}
For every $r,s\ge0$,
\[
\rlt(H_{r,s})=\lambda(H_{r,s})=\frac14.
\]
\end{corollary}

\begin{proof}
The graph contains the monotone path $a_0<b_0<c_0$ and is contained in an ordered blow-up of $\Pvec_3$.  Apply Theorem~\ref{thm:path-sandwich} with $k=3$.
\end{proof}

\section{\texorpdfstring{Sharp template value for $Q_{2,b}$}{Sharp template value for Q2b}}\label{sec:q2b}

For $b\ge1$, let $Q_{2,b}$ be the ordered graph on $1<2<\cdots<b+3$ with edges
\[
12,\quad 13,\quad 23,\quad 34,\quad45,\quad\ldots,\quad(b+2)(b+3).
\]
Thus $Q_{2,b}$ is a transitive ordered triangle followed by a monotone tail of length $b$.

\begin{theorem}\label{thm:lambda-q2b}
For every $b\ge1$,
\[
\lambda(Q_{2,b})=\frac12.
\]
\end{theorem}

\begin{proof}
The lower bound is Corollary~\ref{cor:nonbip-lower}, since $Q_{2,b}$ contains a triangle.

For the upper bound, let $\mathcal T=(X,\mu,A)$ be a template of weight larger than $1/2$.  We show that $\mathcal T$ supports $Q_{2,b}$.  If $A$ contains a loop $(x,x)$, then the constant map from $V(Q_{2,b})$ to $x$ is a support map, so assume that $A$ is loopless.

Fix the directed graph $A$ on $X$.  The support-minimal simplex argument below is a directed variant of the quadratic optimization viewpoint of Motzkin and Straus~\cite{MS65}.  Since the simplex of probability measures on $X$ is compact, the quadratic form
\[
P(\nu)=\sum_{(x,y)\in A}\nu(x)\nu(y)
\]
has a maximizer.  Choose a maximizer $\mu$ whose positive support
\[
S=\{x:\mu(x)>0\}
\]
is inclusion-minimal among maximizing measures.  Its value $p=P(\mu)$ is still larger than $1/2$.  The measure $\mu$ is an interior maximizer on the face $\Delta_S$.  Hence, for every perturbation $h$ supported on $S$ with $\sum_{x\in S}h(x)=0$, the first variation vanishes:
\[
0=\left.\frac{d}{dt}\right|_{t=0} P(\mu+th)
 =\sum_{x\in S}h(x)\bigl(o(x)+i(x)\bigr),
\]
where
\[
o(x)=\mu(N^+(x)),\qquad i(x)=\mu(N^-(x)).
\]
Thus $o(x)+i(x)$ is constant on $S$.  Averaging this constant against $\mu$ gives
\[
o(x)+i(x)=2p\qquad (x\in S). \tag{1}
\]
Because $p>1/2$, no vertex of $S$ is a sink in $A[S]$: if $o(x)=0$, then, since $A$ is loopless and $\mu(x)>0$, we have $i(x)\le1-\mu(x)<1$, contradicting (1).

We claim that $A[S]$ contains a transitive triangle, i.e. vertices $x_0,x_1,x_2$ with
\[
x_0\to x_1,
\qquad x_1\to x_2,
\qquad x_0\to x_2.
\]
Suppose not.  For every directed edge $x\to y$, the sets $N^+(x)$ and $N^+(y)$ are disjoint; otherwise a common out-neighbour would form a transitive triangle with $x\to y$.  Similarly $N^-(x)\cap N^-(y)=\emptyset$.  Therefore, for every edge $x\to y$,
\[
o(x)+o(y)\le1,
\qquad
 i(x)+i(y)\le1. \tag{2}
\]
Multiplying the first inequality in (2) by $\mu(x)\mu(y)$ and summing over all directed edges gives
\[
\sum_x \mu(x)o(x)^2+
\sum_x \mu(x)o(x)i(x)\le p.
\]
The second inequality gives
\[
\sum_x \mu(x)o(x)i(x)+
\sum_x \mu(x)i(x)^2\le p.
\]
Adding and using (1),
\[
4p^2=\sum_x \mu(x)(o(x)+i(x))^2\le2p,
\]
which contradicts $p>1/2$.  Hence a transitive triangle exists.

Let $x_0\to x_1\to x_2$ with $x_0\to x_2$ be such a triangle.  Since $S$ has no sink, there is a directed walk of length $b$ starting at $x_2$,
\[
x_2=y_0\to y_1\to\cdots\to y_b.
\]
The labels $y_j$ are allowed to repeat: a template support map is not required to be injective.  The map
\[
1\mapsto x_0,
\quad
2\mapsto x_1,
\quad
3\mapsto x_2,
\quad
3+j\mapsto y_j\quad(1\le j\le b)
\]
supports $Q_{2,b}$.  Thus every $Q_{2,b}$-avoiding template has weight at most $1/2$.
\end{proof}

\section{\texorpdfstring{The intrinsic role split for $Q_{2,b}$}{The intrinsic role split for Q2b}}\label{sec:roles}

Fix $b\ge1$.  For an ordered graph $G$, define
\[
T_b(G)=\{z:\exists z=v_0<v_1<\cdots<v_b\text{ with }v_{i-1}v_i\in E(G)\text{ for }1\le i\le b\},
\]
and
\[
R_\triangle(G)=\{z:\exists x<y<z\text{ with }xy,xz,yz\in E(G)\}.
\]
Thus $T_b(G)$ is the set of vertices that start a monotone tail of length $b$, and $R_\triangle(G)$ is the set of vertices that appear as the right endpoint of an ordered triangle.  Then $G$ contains $Q_{2,b}$ if and only if
\[
T_b(G)\cap R_\triangle(G)\ne\emptyset.\tag{2}
\]
Indeed, the vertex in the intersection is the right endpoint of the triangle and the start of the $b$-edge tail.

\begin{lemma}[Intrinsic role decomposition]\label{lem:role-split}
Let $b\ge1$, let $G$ be $Q_{2,b}$-free, and put
\[
T=T_b(G),\qquad R=V(G)\setminus T.
\]
Then:
\begin{enumerate}[label=(\roman*)]
\item there are no forward edges from $R$ to $T$;
\item $G[T]$ is ordered-triangle-free;
\item $G[R]$ contains no monotone copy of $\Pvec_{b+1}$.
\end{enumerate}
\end{lemma}

\begin{proof}
If $x\in R$, $y\in T$, $x<y$, and $xy\in E(G)$, then $y$ starts a monotone tail
\[
y=u_0<u_1<\cdots<u_b.
\]
For $b=1$, the edge $xy$ itself shows that $x\in T$.  For $b\ge2$, the vertices
\[
x<y=u_0<u_1<\cdots<u_{b-1}
\]
form a monotone tail of length $b$ starting at $x$.  In both cases this contradicts $x\in R$.  This proves (i).  If $G[T]$ contained an ordered triangle with right endpoint $z\in T$, then $z\in R_\triangle(G)\cap T_b(G)$, contradicting (2).  This proves (ii).  If $G[R]$ contained a monotone copy of $\Pvec_{b+1}$, then its first vertex would start a monotone tail of length $b$, and hence would belong to $T$, contradiction.
\end{proof}

For a level $\ell$, let $TT_\ell,TR_\ell,RT_\ell,RR_\ell$ denote the number of level-$\ell$ ordered pairs with the indicated roles on the left and right endpoint.  By Lemma~\ref{lem:role-split}, no actual edge is counted by $RT_\ell$.  Put
\[
\beta_b=\frac{b-1}{2b}.
\]
Define the $b$-role capacity
\[
C^{(b)}_\ell(T,R)=\frac{TR_\ell+\frac12TT_\ell+\beta_b RR_\ell}{\tau_{\ell,d}},
\]
and the two excess terms
\[
X^T_\ell=\frac{\left(e_\ell(G[T])-\frac12TT_\ell\right)_+}{\tau_{\ell,d}},
\qquad
X^R_\ell=\frac{\left(e_\ell(G[R])-\beta_bRR_\ell\right)_+}{\tau_{\ell,d}}.
\]

\begin{lemma}[Role-capacity decomposition]\label{lem:role-capacity}
Let $b\ge1$.  For every level $\ell$ of a $Q_{2,b}$-free graph,
\[
d_\ell(G)\le C^{(b)}_\ell(T,R)+X^T_\ell+X^R_\ell.
\]
Consequently, if $d_\ell(G)\ge1/2+\eps$, then at least one of
\[
C^{(b)}_\ell(T,R)\ge\frac12+\frac\eps3,
\qquad
X^T_\ell\ge\frac\eps3,
\qquad
X^R_\ell\ge\frac\eps3
\]
holds.
\end{lemma}

\begin{proof}
Since $RT$-edges are absent,
\[
e_\ell(G)=e_\ell(T,R)+e_\ell(G[T])+e_\ell(G[R]).
\]
Also $e_\ell(T,R)\le TR_\ell$, and the definitions of $X^T_\ell,X^R_\ell$ give
\[
e_\ell(G[T])\le \frac12TT_\ell+X^T_\ell\tau_{\ell,d},
\qquad
e_\ell(G[R])\le \beta_bRR_\ell+X^R_\ell\tau_{\ell,d}.
\]
Divide by $\tau_{\ell,d}$.  The alternative follows by pigeonholing.
\end{proof}

\section{Smoothing lemmas}\label{sec:smoothing}

\begin{proposition}[Two-role capacity count bound]\label{prop:role-capacity-bound}
Fix $b\ge1$ and put $\beta_b=(b-1)/(2b)$.  For every role colouring $\chi:\{0,1\}^d\to\{T,R\}$ and every $\eps>0$,
\[
\left|\left\{\ell:
\frac{TR_\ell(\chi)+\frac12TT_\ell(\chi)+\beta_bRR_\ell(\chi)}{\tau_{\ell,d}}
\ge \frac12+\eps
\right\}\right|
\le \frac1{4\eps^2}.
\]
\end{proposition}

\begin{proof}
Let $f$ be the indicator of the role $R$.  In a level-$\ell$ split cell, write $a$ and $c$ for the densities of $R$ in the left and right halves.  The cell capacity equals
\[
(1-a)c+\frac12(1-a)(1-c)+\beta_bac
=\frac12+\frac12(c-a)-\frac1{2b}ac
\le\frac12+\frac12(c-a).
\]
After averaging over level-$\ell$ cells,
\[
S_\ell-\frac12\le\frac12D_\ell,
\]
where
\[
D_\ell=\mathbb P(f(X)=1\mid X_\ell=1)-\mathbb P(f(X)=1\mid X_\ell=0)
=2\mathbb E f(X)(2X_\ell-1).
\]
If $L$ is the set of levels with $S_\ell\ge1/2+\eps$, then $D_\ell\ge2\eps$ on $L$.  Thus, with $m=|L|$,
\[
2\eps m\le\sum_{\ell\in L}D_\ell
=2\mathbb E f(X)\sum_{\ell\in L}(2X_\ell-1)
\le 2\mathbb E\left(\sum_{\ell\in L}(2X_\ell-1)\right)_+
\le \sqrt m,
\]
where the last inequality uses symmetry and Cauchy--Schwarz.  Hence $m\le1/(4\eps^2)$.
\end{proof}

We need a local version of finite-template smoothing.  Let $A\subseteq Q\times Q$ be a fixed template.  Given a labelling $\phi$ of the leaves by $Q$, let $P_\ell$ be the level-$\ell$ density of the template graph.  For a level-$\ell$ parent cell $J$, let $M_J$ be the label distribution inside $J$, and define
\[
K_\ell=\mathbb E_J\sum_{(q,q')\in A}M_J(q)M_J(q').
\]

\begin{lemma}[Local-capacity smoothing]\label{lem:local-template-smoothing}
For every finite $Q$, every template $A\subseteq Q\times Q$, every labelling of the binary leaves by $Q$, and every $\eps>0$,
\[
\left|\{\ell:P_\ell\ge K_\ell+\eps\}\right|
\le \frac{9|Q|}{\eps^2}.
\]
\end{lemma}

\begin{proof}
Fix a level-$\ell$ parent cell $J=J_0\cup J_1$.  Let $\mu_0,\mu_1\in\Delta(Q)$ be the label distributions in the two children, and put
\[
m=\frac{\mu_0+\mu_1}{2},
\qquad
\delta=\frac{\mu_1-\mu_0}{2}.
\]
Then $\mu_0=m-\delta$ and $\mu_1=m+\delta$.  Since the matrix of $A$ has entries in $[0,1]$ and $\|\delta\|_1\le1$,
\[
\mu_0^TA\mu_1
=m^TAm+m^TA\delta-\delta^TAm-\delta^TA\delta
\le m^TAm+3\|\delta\|_1. \tag{3}
\]
Averaging (3) over all level-$\ell$ parent cells gives
\[
P_\ell\le K_\ell+3V_\ell,
\qquad
V_\ell=\mathbb E_J\|\delta_J\|_1.
\]

It remains to bound the total variation energy.  Expose a uniformly random leaf one bit at a time, and let $M_k\in\mathbb R^Q$ be the conditional distribution of its final label after the first $k$ bits are exposed.  Then $(M_k)_{k=0}^d$ is a vector-valued martingale.  At a level-$\ell$ parent cell, the martingale increment is $+\delta_J$ or $-\delta_J$ with equal probability; hence
\[
V_\ell=\mathbb E\|M_\ell-M_{\ell-1}\|_1.
\]
For any set $L$ of $m$ levels, Cauchy--Schwarz, the inequality $\|v\|_1^2\le |Q|\|v\|_2^2$, and martingale orthogonality give
\[
\sum_{\ell\in L}V_\ell
\le
\sqrt{m\sum_{\ell\in L}\mathbb E\|M_\ell-M_{\ell-1}\|_1^2}
\le
\sqrt{m|Q|\sum_{\ell=1}^d\mathbb E\|M_\ell-M_{\ell-1}\|_2^2}
\le \sqrt{m|Q|}.
\]
Indeed, the last sum is at most $\mathbb E\|M_d\|_2^2\le1$.  If $P_\ell\ge K_\ell+\eps$ on every $\ell\in L$, then $V_\ell\ge\eps/3$ on $L$, so
\[
\frac{\eps}{3}m\le\sqrt{m|Q|},
\]
and therefore $m\le9|Q|/\eps^2$.
\end{proof}

\begin{proposition}[Weighted binary $\Pvec_{b+1}$ smoothing]\label{prop:weighted-path-smoothing}
Fix $b\ge1$ and put $\beta_b=(b-1)/(2b)$.  Let $U$ be a subset of a binary interval, and let $H$ be a monotone-$\Pvec_{b+1}$-free graph on $U$.  For every $\eta>0$, the number of levels satisfying
\[
e_\ell(H)\ge\beta_bUU_\ell+\eta\tau_{\ell,d}
\]
is at most
\[
\frac{36(b+1)}{\eta^2}+\frac{b-1}{4b\eta}.
\]
Here $UU_\ell$ is the number of level-$\ell$ ordered pairs with both endpoints in $U$.
\end{proposition}

\begin{proof}
For $b=1$, a monotone-$\Pvec_2$-free graph has no edges, while $\beta_1=0$, so no level satisfies the displayed inequality.  Assume $b\ge2$.

Define the rank of $v\in U$ by
\[
r(v)=\max\{t:\text{there is a monotone path with }t\text{ edges ending at }v\}.
\]
Since $H$ contains no monotone $\Pvec_{b+1}$, we have $r(v)\in\{0,1,\ldots,b-1\}$.  Moreover, if $xy\in E(H)$ with $x<y$, then $r(y)\ge r(x)+1$, and hence $r(x)<r(y)$.

Label vertices of rank $i$ by $i$, and label vertices outside $U$ by $*$.  Consider the template on
\[
Q=\{0,1,\ldots,b-1,*\}
\]
with allowed pairs $(i,j)$ exactly when $0\le i<j\le b-1$.  Let $P_\ell$ be the level-$\ell$ density of this template graph.  Since $E(H)$ is contained in the allowed template graph, it is enough to bound levels with
\[
P_\ell\ge \beta_b\frac{UU_\ell}{\tau_{\ell,d}}+\eta. \tag{4}
\]
By Lemma~\ref{lem:local-template-smoothing}, the levels with $P_\ell\ge K_\ell+\eta/2$ are at most
\[
\frac{9(b+1)}{(\eta/2)^2}=\frac{36(b+1)}{\eta^2}.
\]

For the remaining levels we compare $K_\ell$ to the $\beta_b$-density of $U$.  In a level-$\ell$ parent cell $J=J_0\cup J_1$, let
\[
r_i=\frac{|U\cap J_i|}{|J_i|}\quad(i=0,1),
\qquad
r=\frac{r_0+r_1}{2},
\qquad
\Delta=\frac{r_1-r_0}{2}.
\]
Inside $J$, the local template capacity is at most
\[
\max_{p_0+\cdots+p_{b-1}=r}\sum_{0\le i<j\le b-1}p_ip_j
=\frac12\left(r^2-\min\sum_i p_i^2\right)
\le \frac{b-1}{2b}r^2
=\beta_br^2.
\]
The local normalized value of $UU_\ell$ is $r_0r_1=r^2-\Delta^2$.  Hence
\[
K_\ell-\beta_b\frac{UU_\ell}{\tau_{\ell,d}}
\le \beta_b\mathbb E_J\Delta_J^2. \tag{5}
\]
If (4) holds but $P_\ell<K_\ell+\eta/2$, then the left side of (5) is larger than $\eta/2$, so
\[
\mathbb E_J\Delta_J^2>\frac{\eta}{2\beta_b}.
\]

Finally, $(|U\cap J|/|J|)$ along a random branch is the martingale of conditional expectations of $1_U$, and its level-$\ell$ squared increment is exactly $\Delta_J^2$ averaged over parent cells.  Therefore
\[
\sum_{\ell=1}^d \mathbb E_J\Delta_{\ell,J}^2
\le \operatorname{Var}(1_U)\le\frac14.
\]
The inequality $\mathbb E_J\Delta_J^2>\eta/(2\beta_b)$ can occur for fewer than $\beta_b/(2\eta)$ levels, and hence for at most $(b-1)/(4b\eta)$ levels after harmless rounding.  Combining the two exceptional sets proves the proposition.
\end{proof}

\section{Binary ultrametric cut domination and weighted Mantel}\label{sec:cutdom}

The result in this section is an ultrametric weighted refinement of Mantel's theorem.  Its proof replaces ordinary Zykov symmetrization by a cherry contraction adapted to binary least common ancestors~\cite{Mantel07,Zykov49}.

A \emph{massive binary antichain} is a finite antichain $\mathcal U$ of binary-tree nodes together with positive masses $m_X$ for $X\in\mathcal U$.  If $X,Y\in\mathcal U$, write $X\wedge Y$ for their least common ancestor.  Given a non-negative function $\omega$ on binary-tree nodes, define the pair weight
\[
W(XY)=m_Xm_Y\omega_{X\wedge Y}.
\]
For a graph $H$ on $\mathcal U$, let $W(H)=\sum_{XY\in E(H)}W(XY)$.

\begin{lemma}[Cherry reduction]\label{lem:cherry}
Let $\mathcal U$ be a finite binary antichain with $|\mathcal U|\ge2$.  Then there are distinct $X,Y\in\mathcal U$ such that, with $P=X\wedge Y$,
\begin{enumerate}[label=(\roman*)]
\item no element of $\mathcal U\setminus\{X,Y\}$ lies below $P$;
\item replacing $X,Y$ by $P$ produces another binary antichain;
\item for every $Z\in\mathcal U\setminus\{X,Y\}$,
\[
X\wedge Z=Y\wedge Z=P\wedge Z.
\]
\end{enumerate}
\end{lemma}

\begin{proof}
Take the minimal binary subtree spanned by the nodes of $\mathcal U$, and choose a deepest branching node $P$ of this subtree.  Both child subtrees of $P$ contain an element of $\mathcal U$, and by the choice of $P$ each contains exactly one; call them $X$ and $Y$.  This gives (i).  Since $\mathcal U$ is an antichain, no outside element can contain $P$ or lie below $P$, so replacing $X,Y$ by $P$ preserves the antichain property.  Finally, every path from $X$ or $Y$ to an outside node first exits through $P$, giving the same least common ancestor with that outside node.
\end{proof}

\begin{theorem}[Binary ultrametric cut domination]\label{thm:cutdom}
For every massive binary antichain $\mathcal U$ and every non-negative weight function $\omega$,
\[
\max\{W(H):H\subseteq\binom{\mathcal U}{2}\text{ is triangle-free}\}
\le
\max_{\chi:\mathcal U\to\{0,1\}}
\sum_{\chi(X)\ne\chi(Y)}W(XY).
\]
\end{theorem}

\begin{proof}
We induct on $|\mathcal U|$.  The cases $|\mathcal U|\le2$ are immediate.  Let $H$ be a maximum-weight triangle-free graph on $\mathcal U$.  Choose a cherry pair $X,Y$ using Lemma~\ref{lem:cherry}, and put $P=X\wedge Y$.  For every outside node $Z\in\mathcal U\setminus\{X,Y\}$, the cherry property gives a common kernel
\[
\kappa_Z=\omega_{X\wedge Z}=\omega_{Y\wedge Z}=\omega_{P\wedge Z},
\]
so
\[
W(XZ)=m_Xm_Z\kappa_Z,
\qquad
W(YZ)=m_Ym_Z\kappa_Z. \tag{6}
\]

First suppose $XY\notin E(H)$.  For $V\in\{X,Y\}$ define
\[
D(V)=\sum_{Z\in N_H(V)}m_Z\kappa_Z.
\]
If $D(X)\ge D(Y)$, replace the outside neighbourhood of $Y$ by $N_H(X)$; otherwise replace the outside neighbourhood of $X$ by $N_H(Y)$.  By (6), this does not decrease the weight.  It also preserves triangle-freeness: the copied neighbourhood is independent, because it is the neighbourhood of a vertex in a triangle-free graph, and $X,Y$ remain non-adjacent.

Thus we may assume that $X,Y$ are non-adjacent twins with a common outside neighbourhood $N$.  Replace $X,Y$ by the node $P$ of mass $m_P=m_X+m_Y$, adjacent exactly to $N$.  The resulting smaller graph $H'$ is triangle-free and has exactly the same weight as $H$.  By induction, $H'$ has weight at most a cut of the smaller antichain.  Expanding $P$ back into $X,Y$ with the same colour gives a cut of the original antichain with the same weight.  Hence the desired inequality holds in this case.

Now suppose $XY\in E(H)$.  Then
\[
N_H(X)\cap N_H(Y)=\emptyset,
\]
otherwise a common neighbour would form a triangle with $XY$.  Assume $m_X\ge m_Y$; the other case is symmetric.  Delete $X,Y$ and insert $P$ with residual mass
\[
m_P=m_X-m_Y.
\]
If $m_P=0$, omit $P$.  Join $P$ to the outside set $N_H(X)$ and leave the outside graph unchanged.  The resulting graph $H'$ is triangle-free, since $N_H(X)$ is independent.

By induction, the smaller antichain has a cut of weight at least $W(H')$.  Expand this cut by putting $X$ on the side of $P$ and $Y$ on the opposite side; if $m_P=0$ and $P$ was omitted, put $X$ and $Y$ on opposite sides arbitrarily.  The expanded cut contains the cherry edge $XY$.  For every outside node $Z$, exactly one of $X,Y$ is separated from $Z$, so the expanded cut receives the baseline contribution
\[
m_Ym_Z\kappa_Z.
\]
In addition, whenever the smaller cut separates $P$ from $Z$, it receives the residual contribution
\[
(m_X-m_Y)m_Z\kappa_Z.
\]
Consequently, the expanded cut has weight at least
\[
W(H')+W(XY)+m_Y\sum_{Z\notin\{X,Y\}}m_Z\kappa_Z. \tag{7}
\]
The graph $H'$ has weight
\[
W(H')=
W(H[\mathcal U\setminus\{X,Y\}])
+(m_X-m_Y)\sum_{Z\in N_H(X)}m_Z\kappa_Z.
\]
Combining this with (7), and using the disjointness of $N_H(X)$ and $N_H(Y)$, the expanded cut has weight at least
\[
\begin{aligned}
&W(H[\mathcal U\setminus\{X,Y\}])+W(XY)
 +(m_X-m_Y)\sum_{Z\in N_H(X)}m_Z\kappa_Z
 +m_Y\sum_{Z\notin\{X,Y\}}m_Z\kappa_Z\;\\
&\ge
W(H[\mathcal U\setminus\{X,Y\}])+W(XY)
 +m_X\sum_{Z\in N_H(X)}m_Z\kappa_Z
 +m_Y\sum_{Z\in N_H(Y)}m_Z\kappa_Z\\
&=W(H).
\end{aligned}
\]
This completes the induction.
\end{proof}

\begin{theorem}[Weighted binary Mantel smoothing]\label{thm:weighted-mantel}
Let $U$ be a subset of a depth-$d$ binary interval, and let $H$ be an ordered-triangle-free graph on $U$.  For every $\eta>0$, the number of levels satisfying
\[
e_\ell(H)\ge\frac12UU_\ell+\eta\tau_{\ell,d}
\]
is at most $1/(2\eta)$.
\end{theorem}

\begin{proof}
Let $S$ be the set of levels satisfying the displayed inequality.  Apply Theorem~\ref{thm:cutdom} to the leaf antichain $U$, with all masses equal to $1$, and with level weights
\[
\omega_\ell=
\begin{cases}
1/\tau_{\ell,d},&\ell\in S,\\
0,&\ell\notin S.
\end{cases}
\]
There is therefore a cut $\chi:U\to\{0,1\}$ whose selected-level weighted edge count is at least that of $H$.

Extend $\chi$ to a function $f$ on the whole binary interval by setting
\[
f(x)=
\begin{cases}
 1,&x\in U,
\chi(x)=1,\\
-1,&x\in U,
\chi(x)=0,\\
0,&x\notin U.
\end{cases}
\]
Fix a level-$\ell$ parent cell $J=J_0\cup J_1$, and write $n=|J_0|=|J_1|$.  Let
\[
u_i=\frac{|U\cap J_i|}{n},
\qquad
z_i=\frac1n\sum_{x\in J_i}f(x)
\qquad (i=0,1).
\]
If $a_i$ and $b_i$ are the normalized densities in $J_i$ of colour $1$ and colour $0$, then $u_i=a_i+b_i$ and $z_i=a_i-b_i$.  The normalized number of cut pairs across $J_0,J_1$ is
\[
a_0b_1+b_0a_1
=\frac{u_0u_1-z_0z_1}{2},
\]
while the normalized number of pairs with both endpoints in $U$ is $u_0u_1$.  Hence, after averaging over all level-$\ell$ parent cells,
\[
\frac{\operatorname{cut}_\ell(\chi)-\frac12UU_\ell}{\tau_{\ell,d}}
=-\frac12\mathbb E_J(z_0z_1). \tag{8}
\]
Now put
\[
m=\frac{z_0+z_1}{2},
\qquad
\Delta=\frac{z_1-z_0}{2}.
\]
Then $-z_0z_1=\Delta^2-m^2\le\Delta^2$, so (8) gives
\[
\frac{\operatorname{cut}_\ell(\chi)-\frac12UU_\ell}{\tau_{\ell,d}}
\le \frac12\mathbb E_J\Delta_{\ell,J}^2. \tag{9}
\]
Along a uniformly random branch, the conditional averages of $f$ form a real-valued martingale, and the squared increment at level $\ell$ is exactly $\Delta_{\ell,J}^2$ on the parent cell $J$.  Therefore
\[
\sum_{\ell=1}^d\mathbb E_J\Delta_{\ell,J}^2
\le \mathbb E f^2-(\mathbb E f)^2\le1. \tag{10}
\]
Summing (9) over $\ell\in S$, the total selected-level cut excess above the half-baseline is at most $1/2$.  Since the cut dominates $H$ in the selected weighted sum, while every $\ell\in S$ contributes at least $\eta$ for $H$, we get
\[
\eta |S|\le\frac12.
\]
This proves the theorem.
\end{proof}

\section{Proof of the exact value}\label{sec:q2b-proof}

\begin{theorem}[Rich-level forcing for $Q_{2,b}$]\label{thm:q2b-rich-forcing}
For every $b\ge1$ and every $\eps>0$ there exists $C=C(b,\eps)$ such that every binary-level graph with at least $C$ levels of density at least $1/2+\eps$ contains $Q_{2,b}$.
\end{theorem}

\begin{proof}
Fix $b\ge1$, and let $G$ be a $Q_{2,b}$-free binary-level graph.  Put $T=T_b(G)$ and $R=V(G)\setminus T$.  For any level $\ell$ with $d_\ell(G)\ge1/2+\eps$, Lemma~\ref{lem:role-capacity} implies that one of the three alternatives holds:
\[
C^{(b)}_\ell(T,R)\ge\frac12+\frac\eps3,
\qquad
X^T_\ell\ge\frac\eps3,
\qquad
X^R_\ell\ge\frac\eps3.
\]
The first alternative occurs on at most $9/(4\eps^2)$ levels by Proposition~\ref{prop:role-capacity-bound} with threshold $\eps/3$.

For the second alternative, $G[T]$ is triangle-free by Lemma~\ref{lem:role-split}.  Theorem~\ref{thm:weighted-mantel} with $\eta=\eps/3$ bounds the number of levels with $X^T_\ell\ge\eps/3$ by at most $3/(2\eps)$.

For the third alternative, $G[R]$ is monotone-$\Pvec_{b+1}$-free by Lemma~\ref{lem:role-split}.  Proposition~\ref{prop:weighted-path-smoothing} with $\eta=\eps/3$ bounds the number of levels with $X^R_\ell\ge\eps/3$ by at most
\[
\frac{324(b+1)}{\eps^2}+\frac{3(b-1)}{4b\eps}.
\]
Therefore the number of $(1/2+\eps)$-rich levels in any $Q_{2,b}$-free binary-level graph is bounded by a finite function of $b$ and $\eps$, for instance
\[
C(b,\eps)=1+\frac{9}{4\eps^2}+\frac{3}{2\eps}+\frac{324(b+1)}{\eps^2}+\frac{3(b-1)}{4b\eps}.
\]
This proves the contrapositive of the theorem.
\end{proof}

\begin{theorem}[Exact triangle-tail value]\label{thm:q2b-exact}
For every $b\ge1$,
\[
\rlt(Q_{2,b})=\frac12.
\]
\end{theorem}

\begin{proof}
The lower bound follows from Theorem~\ref{thm:template-lower} and Theorem~\ref{thm:lambda-q2b}.  The upper bound follows from Theorem~\ref{thm:q2b-rich-forcing} and the binary rich-level reduction, Theorem~\ref{thm:rich-level}.
\end{proof}

\section{Sharpness inside binary-level hosts}\label{sec:sharpness}

The rich-level threshold $1/2$ cannot be lowered.  The sharp examples are already present inside the binary model used in the proof.

For a leaf $x=(x_1,\ldots,x_d)\in B_d$, write
\[
\sigma(x)=x_1+\cdots+x_d \pmod 2.
\]
Let $G_d^\sharp$ be the following graph on $B_d$, for $d\ge2$.  First keep every pair $xy$ with $\sigma(x)\ne\sigma(y)$.  Then, only on the deepest level $d$, delete exactly half of the sibling edges.  Equivalently, $G_d^\sharp$ is a subgraph of the parity cut, with all parity-cross edges kept on levels $1,\ldots,d-1$ and exactly $2^{d-2}$ of the $2^{d-1}$ parity-cross sibling edges kept on level $d$.

\begin{proposition}[Binary-level sharpness]\label{prop:binary-sharpness}
For every $b\ge1$ and every $d\ge2$, the graph $G_d^\sharp$ is $Q_{2,b}$-free and
\[
d_\ell(G_d^\sharp)=\frac12\qquad(1\le \ell\le d).
\]
Consequently, for every $\alpha<1/2$ and every $C$ there exists a $Q_{2,b}$-free binary-level graph with at least $C$ levels of density larger than $\alpha$.
\end{proposition}

\begin{proof}
The graph $G_d^\sharp$ is bipartite, with parts given by the parity of $\sigma$.  Since $Q_{2,b}$ contains a triangle, no bipartite graph contains an ordered copy of $Q_{2,b}$.  Thus $G_d^\sharp$ is $Q_{2,b}$-free.

It remains to compute the level densities.  Fix $1\le\ell<d$ and a level-$\ell$ parent cell.  Its two children consist of words with a common prefix, followed by a $0$ in the left child and a $1$ in the right child.  The two suffixes after coordinate $\ell$ vary independently over $\{0,1\}^{d-\ell}$.  A left leaf and a right leaf have opposite parity exactly when these two suffixes have the same parity.  Since $d-\ell\ge1$, exactly half of the ordered pairs across the two children satisfy this condition.  Averaging over all parent cells gives $d_\ell(G_d^\sharp)=1/2$ for every $\ell<d$.

At level $d$ there are $2^{d-1}$ sibling pairs, and by construction exactly $2^{d-2}$ are retained.  Hence $d_d(G_d^\sharp)=1/2$.  Finally, given $\alpha<1/2$ and $C$, choose $d\ge C$.  Then $G_d^\sharp$ has $d$ levels of density $1/2>\alpha$ and contains no $Q_{2,b}$.
\end{proof}

This proposition calibrates the upper theorem precisely: the binary rich-level forcing result starts at $1/2+\eps$, and Proposition~\ref{prop:binary-sharpness} shows that no threshold below $1/2$ can work, even before leaving binary-level hosts.  More explicitly, if one tried to replace $1/2$ by some $\alpha<1/2$, then with $\eps=(1/2-\alpha)/2$ the graphs $G_d^\sharp$ are $(\alpha+\eps,C)$-rich and still $Q_{2,b}$-free for arbitrarily large $C$.

\section{A local forcing diagnostic}\label{sec:local}

The proof of Theorem~\ref{thm:q2b-rich-forcing} is necessarily many-level.  Even for the first non-trivial member $Q_{2,2}$, a purely local counting argument sees only a much larger threshold.  The following elementary lemma is a useful diagnostic: a one-shot five-block argument only reaches $3/4$.

\begin{lemma}[Five-partite $3/4$ forcing]\label{lem:three-quarter}
Let $A<B<C<D<E$ be five ordered vertex classes.  If
\[
d(A,B),\ d(A,C),\ d(B,C),\ d(C,D),\ d(D,E)>\frac34,
\]
then the five classes contain an ordered copy of $Q_{2,2}$ with one vertex in each class.
\end{lemma}

\begin{proof}
Let $D^+$ be the set of vertices of $D$ with a neighbour in $E$.  Since $d(D,E)>3/4$, we have $|D^+|>3|D|/4$.  Let $T\subseteq C$ be the set of vertices sending an edge to $D^+$.  If $c\notin T$, then $c$ sends no edge to $D^+$, and hence has $D$-degree less than $|D|/4$.  From $d(C,D)>3/4$ it follows that
\[
\frac{|T|}{|C|}>\frac23. \tag{11}
\]
Write $\theta=|T|/|C|$.  Since $d(A,C)>3/4$, the number of missing $A$-$T$ edges is less than $|A||C|/4$, and therefore
\[
d(A,T)>1-\frac{|C|}{4|T|}=1-\frac1{4\theta}.
\]
The same argument gives $d(B,T)>1-1/(4\theta)$.

Choose $(a,b,t)$ uniformly from $A\times B\times T$.  The probability that at least one of $ab,at,bt$ is missing is less than
\[
\frac14+\frac1{4\theta}+\frac1{4\theta}<1,
\]
by (11).  Thus there are $a\in A$, $b\in B$, and $t\in T$ with $ab,at,bt\in E(G)$.  Since $t\in T$, there exist $d\in D^+$ and $e\in E$ with $td,de\in E(G)$.  These five vertices form $Q_{2,2}$.
\end{proof}

\section{Concluding remarks}

We proved that every ordered triangle-tail $Q_{2,b}$ has relative density exactly $1/2$, and Proposition~\ref{prop:binary-sharpness} shows that this value is sharp already in binary-level hosts.  The proof also explains why finite obstructions are misleading: the local role split can be maintained for a few levels, but the binary smoothing inequalities prevent the necessary role capacities and internal triangle-free/$\Pvec_{b+1}$-free excesses from persisting on unboundedly many levels.  The binary ultrametric cut-domination theorem, weighted binary Mantel smoothing, and weighted binary $\Pvec_{b+1}$ smoothing are independent tools and may apply to other ordered graphs whose forbidden pattern separates a dense core role from a monotone-tail role.

\end{document}